\documentclass{article}
\usepackage{graphicx} % Required for inserting images^

\usepackage{color}

\usepackage{amssymb}
 \usepackage{amsmath}

\newtheorem{thm}{Theorem}

\newtheorem{lem}{Lemma}

\newtheorem{rem}{Remark}

\title{
Exact ensemble controllability for neural differential equations  
via neural interpolation 
}
\author{Martin Gugat}
\date{
Friedrich-Alexander-Universit\"at Erlangen-N\"urnberg (FAU),
Department of Mathematics, 
Lehrstuhl für Dynamics, Control, Machine Learning  and Numerics (Alexander von Humboldt-Professur), 
Cauerstr. 11, 91058 Erlangen, Germany (\texttt{martin.gugat@fau.de})}

\begin{document}

\maketitle

\begin{abstract}

    We study a system that is governed by neural dynamics.
    Neural dynamics are a model for deep neural networks with a large number of
    layers.
    For a differential equation where the right-hand side 
    %a nonlinear source term that  is defined by a concatenation of sigmoidal functions,
    is given by a neural network,
    %of depth $2$.
we analyze the exact ensemble  controllability  of
the system. This property plays an essential role in Machine Learning.
%(see \cite{zuazua2025machine}).
It concerns  the ability to learn  to perform 
different tasks  simultaneously with  a single neural dynamics:
Exact ensemble controllability 
requires that
$N$ different initial states are steered to corresponding $N$ target states
with a single set of control parameters.

We present a constructive solution to the problem.
The construction is based on the solution of a neural interpolation problem. 
 We show that if the right-hand side of the differential equation 
 is given by a neural network of depth two,
 the interpolation problem can be reduced to the solution of a system of linear equations.

\end{abstract}

\section{Introduction}
In \cite{zuazua2025machine}, it is stated that 
'\emph{control theory of dynamical systems offers a powerful framework for tackling challenges in deep neural networks and other machine learning architectures}.'
%In our 
To contribute to this on-going development  we consider a system that is governed by neural dynamics.
The system has the form
\begin{equation}
\label{ode}
    y'(t) =  L \, y(t) + F(t,\, y)
    \end{equation}
where $L$ is a 
%suitable 
linear operator.
%For example, for $L \, y = y_{xx}$ we obtain  dynamics similar to the heat equation
%and the case $L = y_x$ corresponds to dynamics with a transport equation. 
The nonlinear source term $F$ is defined by a concatenation of sigmoidal functions
in the sense of a neural network, 
therefore we can call our system a neural differential equation.
Neural differential equations have received much attention recently,
since they are a continuous model for
the dynamics in neural networks that
are essential in machine learning,
see for example \cite{chen2018neural}, \cite{richter2026generative}.
Neural ordinary differential equations (ODEs) in optimal control are also discussed in  \cite{geshkovski2022turnpike}. 
In \cite{li2026universal}, the approximation of
dynamical systems by neural ODEs is studied.

We consider the exact controllability properties of
the system. These properties play an
essential role in machine learning.
In particular we study exact ensemble properties,
which correspond to the ability to learn 
to perform 
different tasks 
simultaneously with 
a single neural dynamics,
see  for example \cite{zuazua2025machine}:
\emph{'A fundamental difference between this control-theoretic interpretation of supervised learning and classical
control problems lies in the fact that, in supervised learning, the same controls ...
%(W (t), A(t), b(t)) 
must
\textbf{simultaneously} drive the entire ensemble of input data to their corresponding outputs. This makes the problem
one of ensemble, or simultaneous, control, unlike the classical setting, where controls are typically tailored to each
specific initial state and target.'}

Exact (rather than approximate) controllability is 
important for the mathematical analysis of optimal control in machine learning
since it empowers arguments,  that use the 
exact control property to obtain tracking terms that are equal to zero
after finite time. 
In our contribution we provide exact controls that satisfy an a-priori
bound which also  allows to obtain a bound for the corresponding control cost.

In \cite{danhane2024conditions}, also obstruction to uniform ensemble controllability are investigated
for linear systems.
The relation with polynomial interpolation is stated.
Ensemble controllability by Lie algebraic methods
has been studied in  \cite{agrachev2016ensemble},
where
%mostly approximate ensemble controllability
it is shown that for finite ensembles the
exact controllability property is generic
under conditions on the Lie rank.
Neural dynamics are not considered in  \cite{agrachev2016ensemble}.
The approximate ensemble controllability
for the heat equation
has been investigated in \cite{danhane2024ensemble}.

In \cite{zuazua2025machine}, a result about the  approximate ensemble controllability
 for autonomous neural ordinary differential equations is stated (see also \cite{alvarez2024interplay}).
In \cite{9827563},
the authors show  ensemble controllability
for neural ordinary differential equations 
on a dense submanifold.

In our contribution we follow a different approach,
we consider the 
\emph{exact ensemble controllability}.
%
%With our approach,  
We obtain a result on the exact controllability on 
%towards 
desired trajectories.
This requires a careful choice of the desired trajectories
that ensures that they do not intersect.

%assumption on the minimal distance of
%the desired trajectories.

To be precise, let a depth $d_e\in\{1,2,3,\ldots\}$
be given and for $i \in \{1,\, \ldots d_e\}$ define 
\begin{equation}
\label{fidefinition}
F^{(i)}(t,y) = W^{(i)}(t) \,  \sigma ( A^{(i)}(t)\, y + b^{(i)}(t) ),  
\end{equation}
where $A^{(i)}$ and $W^{(i)}(t)$ are linear operators,
and $b^{(i)}$ are vectors.
The functions  $W^{(i)}(t)$ are assumed to be continuously differentiable.
The function  $\sigma$ is a sigmoidal activation function.
Let 
\begin{equation}
\label{fdefinition}
F(y) = F^{(d_e)}(F^{(d_e-1)}( \ldots (F^{(1)}(y) )  \ldots   )).
\end{equation}

In this contribution, to be specific,  we will focus on $\sigma(z)=\tanh(z)$.
The novel construction that  we present can also be implemented for other 
sigmoidal activation function
as used for example in \cite{costarelli2013approximation}.
%\cite{cybenko1989approximation}.
The construction depends on the assumption that we have a right-hand side that
has two layers. 
This allows us to construct localized basis functions for the right-hand sides that 
%are sums of terms 
have their mass concentrated around given points in 
$\mathbb{R}^d$.

Our result on exact ensemble controllability has the following
form (see Theorem \ref{thm3}  in Section \ref{sec:exact}):
Let a time horizon $T$ and $N$ different given initial states
\[y^{j }_0 \in \mathbb{R}^d, \; j \in \{1,\ldots, N\}\]
and $N$ different $N$  states
\[y^{desi, j }_T \in \mathbb{R}^d,, \; j \in \{1,\ldots, N\}\]
be given.
We will show that for the depth $d_e = 2$  there exist controls
$A^{(i)}$,  $b^{(i)}$ and $W^{(i)}(t)$ ($i \in \{1,2\}$) 
such that
for all $j \in \{1,\ldots, N\}$ the  
solution of (\ref{ode}) with the  initial state
$y^{j }_0$ is driven to the desired target state
$y^{desi, j }_T$ at the time $T$.

In order to prove the result,
in Section \ref{sec:neuralinterp}  we will first consider
the (static) problem of neural interpolation,
where the neural networks maps 
$N$ different given  states
\[z^{j }_0\in \mathbb{R}^d, \; j \in \{1,\ldots, N\}\]
exactly to
$N$ desired states
\[z^{desi, j }_T  \in \mathbb{R}^d , \; j \in \{1,\ldots, N\}, \]
that is  there exist parameters 
$A^{(i)}$,  $b^{(i)}$ and $W^{(i)}$ 
such that
for all $k \in \{1,\ldots, N\}$ 
we have the equations 
\begin{equation}
\label{interpolproblem}
F( z^{ k }_0  ) =  z^{desi,k}_T,  
\end{equation}
where $F$ is as in (\ref{fdefinition}),
'meaning that the 
%INN 
$\ldots$ function 
$\ldots$ 
%space 
exactly passes the interpolation points',
see \cite{park2025unifying}.
In our contribution,
in contrast to  \cite{park2025unifying}, 
we do not relax this condition
%in  the training of the neural network, 
but
keep it as an equality constraint.
Recently, in \cite{doi:10.1137/23M1599744},
%in contrast to our  approach,
the authors have used the term
\emph{approximate  universal interpolation property} (UIP)
to describe  simultaneous  approximation (up to an arbitrarily small error $\varepsilon >0$)
of $N$ data points by a suitably chosen function $F$.

The proof of the result on interpolation is based on
auxiliary  functions that are discussed in Section \ref{sec:auxil}.
We show how the neural networks yield  approximate $\delta$-sequences.
Note that the problem of neural interpolation is different from the problem  in
the well-known universal approximation theorem in \cite{cybenko1989approximation}.

%%%%%%%%%%%%%%%%%%%%%%%%%%%%%%%%%%%%%%%%%%%%%%%%%%%%%%%%%%%%%%%%%%%%%%%%%%%%%%%%%%%%%%%%%%%%%%%%%%%%
%Interpolating neural networks are also considered in \cite{park2025unifying}.
% THIS IS JUST A NAME; THEY DO NOT REALLY INTEPOLATE!
%  title={Unifying machine learning and interpolation theory via interpolating neural networks}
%

%%%%%%%%%%%%%%%%%%%%%%%%%%%%%%%%%%%%%%%%%%%%%%%%%%%%%%%%%%%%%%%%%%%%%%%%%%%%%%%%%%%%%%%%%%%%%%%%%%%%%%%%

%%%%%%HERE ID THE NEW TEXT:

For our interpolation result, we assume that $d_e=2$
and show that 
in this case the interpolation problem 
can be reduced to the solution of a linear system
with a regular $N\times N$ matrix.
%has a solution with a regular 
%outer matrix $W^{(2)}$.
%

%%%%%%%%%%%%%%%%%%%%%%%%%%%%%%%%%%%%%%%%%%%%%%%%%%%%%%%%%%%%%%%%%%%%%%%%%%%%%%%%%%%%%%%%%%%%%%%%%%%%%%%%%%%%

%%%%%%%%%%%%% REPLACE THE FOLLOWING SENTENCE IN THE INTRODUCTION:
%For our interpolation result, we assume that $d_e=2$
%and show that 
%in this case the interpolation problem has a solution with a regular 
%outer matrix $W^{(2)}$.
%%%%%%%%%%%%%%%%%%%%%%%%%%%%%%%%%%%%%%%%%%%%%%%%%%%%%%%%%%%%%%%%%%%%%%%%%%%%%%%%%%%%%%%%%%%%%%%%%%%%%%%%%

The following example  shows that for the depth $d_e=1$, 
also for small values of $N$ in general the interpolation problem does not have
a solution with a regular outer matrix $W^{(1)} $.

Let $N=3$.
Assume that 
$ z^{1}_0  = 0 \in \mathbb{R}^d$.
Let 
$ z^{2}_0 \in \mathbb{R}^d$,  $z^{ 2 }_0 \not =0$  be given.
Define  $ z^{3}_0  = -  z^{2}_0  $.
Let $z^{desi, 1 }_T = 0  \in \mathbb{R}^d$ and  let $z^{desi, 2 }_T\not= 0$ be given.
Assume that 
 $z^{desi, 3}_T \not =   -  z^{desi, 2 }_T$. 
For the depth $d_e=1$, we have
\[F(z) = W^{(1)} \,  \sigma ( A^{(1)}\, z + b^{(1)} ). \]
Suppose that the interpolation problem has a solution,
that is  for all $k \in \{1,\, 2, \, 3\}$, equation 
(\ref{interpolproblem}) holds.

The interpolation condition for $k=1$ implies that $F(0)=0$, hence 
   $ \sigma ( b^{(1)} )$  is in the kernel of $ W^{(1)} $.
Then we have
$z^{desi, 2 }_T = F(z^{2}_0  ) = W^{(1)} \,  \sigma ( A^{(1)}\,   z^{2}_0  + b^{(1)} )$
and
$z^{desi, 3 }_T = F(z^{3}_0  ) = W^{(1)} \,  \sigma (  - A^{(1)}\,   z^{2}_0  + b^{(1)} )$.
Our  assumption implies that 
$z^{desi, 2}_T  +  z^{desi, 3 }_T \not=0$, which  is  equivalent to the statement that 
$  \sigma ( A^{(1)}\,   z^{2}_0  + b^{(1)} ) +  \sigma (  - A^{(1)}\,   z^{2}_0  + b^{(1)} )  $
  is not in the kernel of $ W^{(1)} $.
Suppose that $  b^{(1)} =0 $.
Then 
 $\sigma ( A^{(1)}\,   z^{2}_0  + b^{(1)} ) +  \sigma (  - A^{(1)}\,   z^{2}_0  + b^{(1)} ) = 0$,
 which is a contradiction.
 Hence we have shown that  $  b^{(1)} \not =0 $.
 This implies that 
 $ \sigma ( b^{(1)} )\not =0$, and hence  
 the matrix  $W^{(1)} $
 has a kernel that is at least of dimension one.

The paper is organized as follows. 
In Section \ref{sec:auxil},  we introduce auxiliary functions based on the hyperbolic tangent 
and establish their approximation properties. This also  yields an approximate $\delta$-sequence result. 
Section \ref {sec:neuralinterp}  uses these properties to 
establish the solvability of the neural interpolation problem constructively, 
reducing it to the solution of a  linear system of equations.
In Section  \ref{sec:exact}, we apply this interpolation framework to prove the main result 
on exact ensemble controllability. 
Section \ref{sec:exa} illustrates the construction explicitly for an ensemble of two elements.
%%%%%%%%%%%%%%%%%%%%%%%%%%%%%%%%%%%%%%%%%%%%%%%%%%%%%%%%%%%%%%%%%%%%%%%%%%%%%%%%%%%%%%%%%%%%%%%%%%%%

\section{Auxiliary functions to approximate the Dirac delta distribution}
\label{sec:auxil}

First we introduce some auxiliary functions. 
Let a real number $h>0$ be given.
(As an example, take $h=\frac{1}{2}$.)
Define the function
\[  g_0(x) = \frac{\tanh( x + h ) - \tanh(x - h)}{ 2 \tanh(h)} .\]
Then we have
\begin{equation}
\label{2gl}
g_0(x) = \frac{1 + \cosh( 2\, h)}{ \cosh( 2 \, x) + \cosh( 2 \, h)}.
\end{equation}
We also use the notation $g_0(x,\, h)$ to emphasize the dependence on $h$.
Note that
\[\lim_{h\rightarrow 0+ } g_0(x, \, h) =  \frac{2}{ 1+ \cosh( 2 \, x)}.\]

We have the following lemma about the properties of $g_0$:
\begin{lem}
\label{lemma1}
For all $h>0$ we have  $g_0(h)\geq \frac{1}{2}$. 
For $d\geq 2$ and $h\geq \frac{d}{2}$ 
we have 
\begin{equation}
\label{g0dungleichung}
g_0 \left( \frac{h}{d} \right) \geq \frac{d}{d+1}.
\end{equation}
We have   
\begin{equation}
\label{g0abldungleichung}
|g_0'(x)|  \leq 2 \, g_0(x) \,  |\tanh(2\, x)|.
\end{equation}

{
\color{blue}
Moreover, in a neighborhood of zero,
i.e. for $x\in [-h/2, \, h/2]$,
we have
\[|g_0'(x) | \leq  4\, \exp( - h)
\;
\mbox{ and }
\;
g_0(x) \geq  1 - 2\, h \, \exp(-\, h). 
\]
}

\end{lem}
\textbf{Proof:}
%This can be seen as follows:
First for $d\geq 2$ we consider 
the auxiliary function
\[G(d) = \cosh(d) - d\, \cosh(  1 ).\]
Then $G(2) = \cosh(2) - 2\, \cosh(1)  >0$
and  for $d\geq 2 $ we have $G'(d)  = \sinh(d) - \cosh(1) >0$.
Hence for all $d\geq 2$ we have  $G(d) >0$.

Now  for $ x \geq d$ define the auxiliary function
\[H(x) = \cosh(x ) - d\, \cosh \left( \frac{x}{d} \right).\]
Then we have $H( d ) = \cosh(d) - d\, \cosh( 1 ) = G(d) >0$.
Moreover, $H'(x) = \sinh(x) - \sinh \left( \frac{x}{d} \right)
 >0$.
Thus for all $x\geq d$ we have $H(x)>0$
and thus $ d \, \cosh( \frac{x}{d} ) <  \cosh(x)  $.
Hence for all $x \geq d$ we have 
\begin{equation}
\label{hilfsungleichung}
\cosh\left(\frac{x}{d} \right) \leq \frac{1}{d}  \cosh(x).  
\end{equation}
Hence  we have
\[  g_0\left( \frac{h}{d} \right) \geq  \frac{\cosh( 2\, h)}{ \cosh( 2 \, \frac{h}{d}  ) + \cosh( 2 \, h)}
\geq   \frac{\cosh( 2\, h)}{  \frac{1}{d} \cosh( 2 \, h  ) + \cosh( 2 \, h)}  = \frac{d}{d + 1} \]
and   (\ref{g0dungleichung})  follows.
%%%%%%%%%%%%%%%%%%%%%%%%%%%%%%%%%%%%%%%%%%%%%%%%%%%%%%%%%%%%%%%%%%%%%%%%%%%%%%%%%%%%%%%%%%%%%%%%%%%%%%%%%%%%%%%%%%%%%%%%%%%
For the derivative we have   
\[g_0'(x) = -2  \, \frac{1 + \cosh(2\, h)}{(\cosh(2x) + \cosh(2h))^2} \sinh(2x).\]

This yields the inequality
\[|g_0'(x) | \leq
2 \frac{1 + \cosh(2\, h)}{\cosh(2x) + \cosh(2h)} \, \frac{  | \sinh(2x)  | }{\cosh(2x) + \cosh(2h)}
\]
\[
\leq  2 \, g_0(x) \frac{  | \sinh(2x)  | }{\cosh(2x)} 
=  2 \, g_0(x) \, |\tanh(2\, x)|.
\]
Thus we have shown (\ref{g0abldungleichung}). 
%%%%%%%%%%%%%%%%%%%%%%%%%%%%%%%%%%%%%%%%%%%%%%%%%%%%%%%%%%%%%%%%%%%%%%%%%%%%%%%%%%%%%%%%%%%%
%{\textbf{Proof: }}
{\color{blue} 
For the derivative we also have   
\[|g_0'(x)| 
%= |2  \, \frac{1 + \cosh(2\, h)}{(\cosh(2x) + \cosh(2h))^2} \sinh(2x)|\]
%\[
= 
2  \, \frac{1 + \cosh(2\, h)}{\cosh(2x) + \cosh(2h)} \; \left| \frac{\sinh(2x)}{   \cosh(2x) + \cosh(2h) }  \right|
\]
\[=
2 \, g_0(x)  \;  \frac{|\sinh(2x)|}{   \cosh(2x) + \cosh(2h) }            
= 2  \, g_0(x) \,|\tanh(2 x)| \,  \frac{\cosh(2x)}{   \cosh(2x) + \cosh(2h) }     
\]
\[
= 2 \, g_0(x) \,|\tanh(2 x)| \,  \frac{1}{  1  + \frac{ \cosh(2h) }{\cosh(2 x)} }     
\]
This implies that for 
%$d\geq 2$ and  
$x\in [ - h/2, h/2]$ we have 
\[|g_0'(x)|
\leq 
2  \,|\tanh( h )| \,  \frac{1}{  1  + \frac{ \cosh(2h) }{\cosh (h) } }     
\leq
2  \frac{1}{  \frac{ \exp (2h) }{   2 \exp(h) } } 
\leq 4 \, \exp(-\, h).
\]
This implies in turn that for 
$x\in [ - h/2, h/2]$ we have 
\[g_0(x) \geq 1 -
\frac{h}{2}  \,  4 \, \exp(-\, h) = 1 - 2\, h \, \exp(-\, h)
\geq 1 -  2 \exp(-1).
\]
}

%%%%%%%%%%%%%%%%%%%%%%%%%%%%%%%%%%%%%%%%%%%%%%%%%%%%%%%%%%%%%%%%%%%%%%%%%%%%%%%%%%%%%%%
This finishes the proof of Lemma \ref{lemma1}. 
%$\Box$
\vspace*{1cm}
\\
Let  scaling factors $\gamma_1>0$, $\gamma_2>0$, $h_1>0$ and $h_2>0$  be given 
(As an example, take $\gamma_1 =\gamma_2 = 10$, $h_1 = h_2 = \frac{1}{2}$.)
For $y\in \mathbb{R}^d$ and  $y^{desi}$ in  $\mathbb{ R}^d $ define
\[F^{(1)}(y) = \sum_{j = 1}^d g_0( \gamma_1 (y_j -  y_j^{desi}) , \, h_1)  \]
and
\[F^{(2)}(y) = g_0 ( \gamma_2 (  F^{(1)}(y) - d), \, h_2).\]
For $d_e=2$ we have
\[ F(y, y^{desi} ) =  F^{(2)}(  F^{(1)}(y) ) \]
which can approximate a Dirac-delta function with support at the point $y^{desi}$.

We introduce the notation 
\begin{equation}
\label{g0definition}
G_0(y,    y^{desi}) =  F(y, y^{desi} )  =  F^{(2)}(  F^{(1)}(y) ). 
\end{equation}

For $z \in \mathbb{R}^d$ let 
$\|\cdot\|_\infty  = \max_{ j\in \{1\ldots, d \}   } \, |z_j|$ 
denote the maximum-norm in $\mathbb{R}^d$.
%with the following property:
%For all  $ j\in \{1\ldots, d \} $ we have 
%\[    |z_j|\leq \|z\|. \]

We have the following lemma about the properties of $G_0$:
\begin{lem}
\label{lemma2}
For all $y\in \mathbb{R}^d$ we have 
$G_0(  y  ,\,   y^{desi}) \geq 0$ and
$G_0(  y^{desi}  ,\,   y^{desi}) = 1$.

Let $\lambda \in (0, \, \infty) $ and $\Delta >0$  be given. 

Assume that   $ h_1 = \lambda\,  \Delta \, \gamma_1  \geq \frac{d}{2}$
and
\begin{equation}
    \label{h2defi2004}
    h_2 =   \frac{d}{d+1} \,    \left( 1 -    \frac{d}{ \cosh( 2\,   \lambda \,  \Delta \, \gamma_1 )} \right)  \, \gamma_2.
    \end{equation}

\begin{enumerate}
    \item 
     Let $\mu > 1$  be given.
Assume that 
$\gamma_1$ is sufficiently large such that 
%there exists a number $\kappa(\gamma_1) >0$ such that 
\begin{equation}
\label{kappadefinition}
\kappa(\gamma_1) = 
\frac{d+1}{d}  \frac{1}{ 1 -    \frac{d}{ \cosh( 2\,   \lambda \,   \Delta \gamma_1) } }
 \left[
 1 -  4 \exp( - 2 (\mu - 1) \, \lambda \,\Delta \, \gamma_1)
 \right]
-1 >0.
\end{equation}
Note that for $\lambda \in (0, \, \infty)$ we have 
\[
\lim_{\gamma_1 \rightarrow \infty}  \kappa(\gamma_1) =  \frac{1}{d}.
\]
    
%Let $\Delta >0$ be given such that 
%\begin{equation}
%\label{bed1}
%  2 \, \gamma \, {\Delta} \geq \textrm {ln}\left(  4 + 4\,  \cosh( 2\, h)  \right) .
%  \end{equation}
 If   $ \| y -   y^{desi}\|_\infty \geq \mu \, \lambda \, \Delta$,
  we have 
  
\begin{equation}
\label{ungleichun202501}
G_0(y,    y^{desi}) \leq  4 \, \exp\left( - 2 \, \kappa(\gamma_1) \, h_2 \right).
\end{equation}
  
%\begin{equation}
%\label{ungleichun202501}
%    G_0(y,    y^{desi}) \leq 
%\left(  2  + 2\,  \cosh( 2\, h) \right)  \, \exp\left(-  \gamma  \right).
%
%   G_0(y,    y^{desi}) \leq 
% 4  \, \exp\left(-2 ( 1 - \lambda)   \gamma  \right).
%\end{equation}
\item 
 If 
 $ \| y -   y^{desi}\|_\infty \leq  \frac{ \lambda  \, \Delta}{d}$,
  we have 
\begin{equation}
\label{ungleichun202502}
G_0(y,    y^{desi}) 
\geq
\frac{1}{2}.
\end{equation}
and thus
$1-  G_0(y,    y^{desi}) \leq  \frac{1}{2}$.
%
%\begin{equation}
%\label{ungleichun202502}
%   1-  G_0(y,    y^{desi}) \leq 
%   \frac{ \cosh\left( 4 \, d \,  \lambda \, \gamma^2\, \Delta  \,  \tanh(2\,  \gamma\,  \lambda \,  \Delta ) \right)   -1  }{  2 }
%\end{equation}
%and thus
%\[
%G_0(y,    y^{desi}) 
%\geq
%1 -    \frac{ \cosh\left( 4 \, d \,  \lambda \, \gamma^2\, \Delta  \,  \tanh(2\,  \gamma\,  \lambda \,  \Delta ) \right)   -1  }{  2 }
%\geq 1 -   4 \, d^2 \,  \lambda^2 \, \gamma^4\, \Delta^2  \,  \tanh^2(2\,  \gamma\,  \lambda \,  \Delta ) 
%\]
%Moreover, if  $ h \geq  \lambda \, \gamma \,  \Delta$,   we have 
%\[
%g_0( \gamma (y_j -  y_j^{desi}) ) \geq
%\frac{ \cosh( 2\,   \lambda \, \gamma \,  \Delta )}{ \cosh(2 \,  \gamma\, \lambda \,  \Delta) + \cosh( 2 \,\lambda \, \gamma \,  \Delta   )}
%= \frac{1}{2}.
%\]

%\item 
%We have   $|g_0'(x)|  \leq 2 \, g_0(x) \,  |\tanh(2\, x)|$.
% If  there exists $\lambda >0$ such that 
% $ \| y -   y^{desi}\| \leq   \lambda  \, \Delta$,
%  we have 
%\begin{equation}
%\label{ungleichun202503}
%   G_0'(y,    y^{desi}) \leq 
%{ \cosh\left( 4 \, d \,  \lambda \, \gamma^2\, \Delta  \,  \tanh(2\,  \gamma\,  \lambda \,  \Delta ) \right)   -1  }.
%\end{equation}

\end{enumerate}
\end{lem}
\textbf{Proof:}

%%%%%%%%%%%%%%%%%%%%%%%%%%%%%%%%%%%%%%%%%%%%%%%%%%%%%%%%%%%%%%%%%%%%%%%%%%%%%%%%%%%%%%%%%%%%%%%%%%%%%%%%%%%%%%%%%%%%%%%%%%%
\emph{1.}
If   $ \| y -   y^{desi}\|_\infty \geq \mu\, \lambda \,\Delta$,
there exists an index  $j \in \{1, \ldots, N\}$ such that
$  |y_j -  y_j^{desi}|\geq \mu \, \lambda \, \Delta $, and hence 
we have
\[
g_0( \gamma_1 (y_j -  y_j^{desi}) , \, h_1) \leq 
g_0(\, \gamma_1 \,\mu \, \lambda \,\Delta, \, h_1) 
=
\frac{1 + \cosh( 2\, h_1)}{ \cosh(  2 \, \mu \, \lambda   \, \gamma_1\,\Delta) + \cosh( 2 \, h_1)}
\]
\[
\leq
 \frac{2  + 2\,  \cosh( 2\, h_1)}{ \exp( 2\, \mu \,   \lambda  \, \gamma_1 \,\Delta)}
 =
 \left(  2  + 2\,  \cosh( 2\, h_1) \right)  \, \exp(- 2 \, \mu  \,  \lambda \, \gamma_1\,\Delta).
\]
%\[
%=
% \frac{1 + \cosh( 2\, h)}{ 2 \cosh( \gamma\,\Delta + h) \, \cosh(\gamma\,\Delta - h   )}.
%\]
Hence we have
\[ d -  \sum_{k = 1}^d g_0( \gamma_1 (y_k -  y_k^{desi}), \, h_1 )
%\geq  1 - \left(  2  + 2\,  \cosh( 2\, h_1)  \right) \, \exp(- 2 \, \gamma_1\,\Delta).
\geq  1 -  4 \exp( - 2 (\mu - 1)  \, \lambda  \, \gamma_1 \, \Delta). 
\]
This implies in turn the inequality
\[G_0(y,    y^{desi}) \leq 
 g_0 ( \gamma_2 ( d -  \sum_{k = 1}^d g_0(\gamma_1 (y_k -  y_k^{desi}) )), \, h_2)
 \]
 \[
 \leq \frac{  1 +   \cosh( 2\, h_2) }{\cosh\left(2 \, \gamma_2 \,   
 \left[  1 -  4 \exp( - 2 (\mu - 1)  \,  \lambda  \, \gamma_1 \, \Delta) \right] 
 \right) +  \cosh (2 \, h_2) }.
 \]
%%%%%%%%%%%%%%%%%%%%%%%%%%%%%%%%%%%%%%%%%%%%%%%%%%%%%%%%%%%%%%%%%%%%%%%%%%%%%%%%%%%%%%%%%%%%%%%%%%%%%%%%
%  \[
% \leq \frac{  1 +   \cosh( 2\, h_2) }{\cosh\left(2 \, \gamma_2 \,   \left[1 - \left(  2  + 2\,  \cosh( 2\, h_1)  \right) \, \exp(- 2 \, %\gamma_1\,\Delta) \right]      \right) +  \cosh (h_2) }.
%\]
%%%%%%%%%%%%%%%%%%%%%%%%%%%%%%%%%%%%%%%%%%%%%%%%%%%%%%%%%%%%%%%%%%%%%%%%%%%%%%%%%%%%%%%%%%%%%%%%%%%%%%%%
With $h_2$ as in (\ref{h2defi2004}), i.e. 
\[h_2 =  \gamma_2 \, \frac{d}{d+1} \   \left( 1 -    \frac{d}{ \cosh( 2\,   \lambda \, \gamma_1 \,  \Delta )} \right),\]
this implies that
%\[G_0(y,    y^{desi}) 
%\leq 
% g_0 ( \gamma_2 ( d -  \sum_{j = 1}^d g_0(\gamma_1 (y_j -  y_j^{desi}) )))
% \]
 \[
 G_0(y,    y^{desi}) \leq \frac{ 2\,  \exp(2\, h_2)  }{ \cosh\left(2 \, 
 \frac{d+1}{d}  \frac{1}{ 1 -    \frac{d}{ \cosh( 2\,   \lambda \, \gamma_1 \,  \Delta ) } }
 \, h_2
 \,   
 \left[
 1 -  4 \exp( - 2 \,(\mu - 1) \lambda  \, \gamma_1 \, \Delta)
 \right]
 \right)  }.
 \]
Define the  number $\kappa$ as
\[
\kappa = 
\frac{d+1}{d}  \frac{1}{ 1 -    \frac{d}{ \cosh( 2\, (\mu - 1)    \lambda \, \gamma_1 \,  \Delta ) } }
 \left[
 1 -  4 \exp( - 2 (\mu - 1) \, \lambda  \, \gamma_1 \, \Delta)
 \right]
-1.
\]
If $\kappa >0$, this  implies that
\[G_0(y,    y^{desi}) \leq  4 \, \exp\left( - 2 \, \kappa \, h_2 \right).
\]
Thus we have shown (\ref{ungleichun202501}).

\emph{2.}  Assume that   $ \| y -   y^{desi}\|_\infty \leq   \frac{\lambda  \, \Delta}{d}$.
Since  $ h_1 = \lambda \, \gamma_1 \,  \Delta\geq \frac{d}{2}$,   for all $j \in \{1, \ldots, N\}$    we have
\[
g_0( \gamma_1 (y_j -  y_j^{desi}), \, h_1 ) \geq
\frac{  1+ \cosh( 2\,   \lambda \, \gamma_1 \,  \Delta )}{ \cosh(2 \,  \frac{1}{d} \, \gamma_1\, \lambda \,  \Delta) + \cosh( 2 \,\lambda \, \gamma_1 \,  \Delta   )}
= \frac{ 1 + \frac{1}{ \cosh( 2\,   \lambda \, \gamma_1 \,  \Delta ) }   }{  1  +  \frac{   \cosh(2 \,  \frac{1}{d} \, \gamma_1\, \lambda \,  \Delta) }{ \cosh( 2 \,\lambda \, \gamma_1 \,  \Delta   ) }     }.
\]
Thus
%(\ref{g0dungleichung})
(\ref{hilfsungleichung}) 
implies 
\[
g_0( \gamma_1 (y_j -  y_j^{desi}), \, h_1 )
\geq
 \frac{ 1 + \frac{1}{ \cosh( 2\,   \lambda \, \gamma_1 \,  \Delta ) }   }{  1  +   \frac{1}{d}       }
 =
 \frac{d}{d+1}
\left(    1 + \frac{1}{ \cosh( 2\,   \lambda \, \gamma_1 \,  \Delta ) }     \right).
\]

%If $ h_1 = \lambda \,  \,  \Delta$,   for all $j \in \{1, \ldots, N\}$    we have
%\[
%g_0( \gamma (y_j -  y_j^{desi}) ) \geq
%\frac{  1+ \cosh( 2\,   \lambda \,  \Delta )}{  \gamma\, \cosh(2 \,  \lambda \,  \Delta) + \cosh( 2 \,\lambda \,   \Delta   )}
%\geq  \frac{ 1 + \frac{1}{   \gamma \, \cosh( 2\,   \lambda \, \gamma \,  \Delta ) }   }{2}.
%\]

Hence  we have
\[ d -  \sum_{j = 1}^d g_0( \gamma_1 \,  (y_j -  y_j^{desi}), \, h_1 )
%\leq   \frac{1 -    \frac{1}{ \cosh( 2\,   \lambda \, \gamma \,  \Delta ) }  }{2} \, d
\leq \frac{d}{d+1}  \left( 1 -  \frac{d}{ \cosh( 2\,   \lambda \, \gamma_1 \,  \Delta ) } \right).
\]
This implies in turn the inequality
\[G_0(y,    y^{desi}) \geq 
 g_0 \left(  \gamma_2\,  \frac{d}{d+1}  \left( 1 -  \frac{d}{ \cosh( 2\,   \lambda \, \gamma_1 \,  \Delta ) } \right) , \, h_2  \right).
 \]
With $h_2$ as in definition  (\ref{h2defi2004}), that is 
\[
h_2 =  \gamma_2 \, \frac{d}{d+1} \   \left( 1 -    \frac{d}{ \cosh( 2\,   \lambda \, \gamma_1 \,  \Delta )} \right), 
\]
using Lemma \ref{lemma1} we obtain 
 \[
G_0(y,    y^{desi})  \geq g_0 \left(  h_2, \, h_2 \right) 
\geq 
\frac{1}{2},
\]
hence (\ref{ungleichun202502}) holds.
This finishes the proof of Lemma \ref{lemma2}.

%we obtain
%\[  |\frac{d}{dy} G_0(y,    y^{desi})| \leq  |\sinh(2\,  )|
%%
%\]

%Figure \ref{fig1}  shows the function $F$ for $d = 2$, $ d_e=2$, $h_1 = h_2 =  \tfrac{1}{2}$, $ \gamma_1 = \gamma_2 = 10$
%and $y^{desi} = ( 0, \, 0)^\top$.
%Note that 
%$F(y^{desi})=1$ and
%for $y\not=  y^{desi}$, 
%the value of $F$ decays rapidly.

%Figure  \ref{fig1alt}   shows the function $F$ for $d = 2$, $ d_e=2$, $h_1 =    \tfrac{1}{2} $,  
%$h_2 = 6$,  $\gamma_1 = \gamma_2 = 10$
%and $y^{desi} = ( 0, \, 0)^\top$.
%In this case in a neighborhood  of $y^{desi}$, the graph is very flat. 

%\begin{figure}[htbp]

%		\includegraphics[width=\textwidth]{bilbaobil1untitled.png}
%		\caption{An example for $F(y) $ with $d = 2$,  $\gamma_1 = \gamma_2 = 10$ and $h_1 = h_2 =\frac{1}{2}$.  }
%		\label{fig1}
	
%\end{figure}

%\begin{figure}[htbp]
%		\includegraphics[width=\textwidth]{bilbaobild2.png}
%		\caption{An example for $F(y) $ with $d = 2$, $\gamma_1= \gamma_2=  10$,  
 %         $h_1  =\frac{1}{2}$
 %       and $h_2=6$.}
%		\label{fig1alt}	
%\end{figure}

\begin{rem}
\label{remark1}
%\textbf{Remark:} 
The functions $G_0$ depend on the parameters
$\gamma_1$, $\gamma_2$, $h_1$ and $h_2$.
We introduce the notation 
$G_0(x, y,  \gamma_1, \,  h_1, \,  \gamma_2, \,  h_2 ) $
to emphasize the dependence of $G_0$ on the parameters.
The parametric family 
$ \left(G_0(\cdot, \, \cdot, \,  \gamma_1, \gamma_2, h_1, h_2)\right)_{ \gamma_1, \gamma_2, h_1, h_2 }$
can be used to approximate $\delta$-functions,
that is to obtain an approximate $\delta$-sequence
on the set 
\[B_M(y)  =\{x \in \mathbb{R}^d: \|x - y\|_\infty \leq M\},\]
where $M>0 $ is a real number.  This is shown in the following subsection.
\end{rem}

\subsection{An approximation result}
Let a real number $M>0$ be given.
We can use a normalized $\tilde G_0$ with 
 \[   \tilde G_0(x, \, y ) = \frac{1}{\int_{B_M(y)} G_0(z, \, y) \, dz } \, G_0(x, y),\]
 (thus $\int_{B_M(y)} \tilde G_0(x,  \, y) \, dx =1$)
as a convolution kernel for an approximate $\delta$-sequence
%approximation
in the following sense:

\begin{thm}
    \label{convolution}
Let $K \subset \mathbb{R}^d$ be compact and  a continuous function 
\[f:  K \rightarrow (-\infty,\, \infty)\]
be given. 
Choose $M>0$ sufficiently large.

There exists a sequence of parameters
$(\gamma_1^{(n)}, \,  h_1^{(n)}, \,  \gamma_2^{(n)}, \,  h_2^{(n)})_{n=1}^\infty$
with
$\gamma_1^{(n)}>0$,   $h_1^{(n)}>0$,  $\gamma_2^{(n)}>0$,   $h_2^{(n)}>0$ 
such that 
%$h_1 = \lambda \, \gamma_1 \,  \Delta$  
%and
% $h_2$ as defined in  (\ref{h2defi2004})  we have 
\[
%\lim_{\gamma_1 \rightarrow \infty,  \, \gamma_2\rightarrow \infty }  
\lim_{n\rightarrow \infty}
\max_{y\in K} \left| f(y)  -   \int_{B_M(y)} f(x) \,   \tilde G_0\left(x, \, y, \,\gamma_1^{(n)}, \,  h_1^{(n)}, \,  \gamma_2^{(n)}, \,  h_2^{(n)}    \right) \, dx \right| = 0.
\]
\end{thm}
\begin{rem}
Theorem \ref{convolution} is closely related to  the
well-known universal 
uniform approximation property, see \cite{cybenko1989approximation}.
Multivariate hyperbolic tangent neural network approximation
(not interpolation) has also been considered in \cite{anastassiou2011multivariate}.
%Note that for more regular functions $f$, an improved  order of approximation 
%can be shown using the smoothness of $f$, see \cite{anastassiou2011multivariate}, Theorem 8.
\end{rem}
\textbf{Proof:}
The proof uses standard arguments.
For the convenience of the reader we include it here. 
We show that $\tilde G_0$ yields an
approximate $\delta$-sequence,
since  we have the following three properties:
\begin{enumerate} 
\item 
For all $x$, $y\in \mathbb{R}^d$ we have 
$\tilde G_0(x, \, y) \geq 0$.

\item 
Let 
$ \lambda  \,  \Delta  := \frac{h_1}{\gamma_1} 
\geq \frac{d}{2 \,\gamma_1}$ and 
$h_2$ be defined as in (\ref{h2defi2004}), that is
\begin{equation}
   % \label{h2defi2004}
    h_2 =   \frac{d}{d+1} \,    \left( 1 -    \frac{d}{ \cosh( 2\,  h_1 )} \right)  \, \gamma_2.
    \end{equation}

Then
due to Lemma \ref{lemma2}, 2. we have
\[( 2\,M) ^d \geq \int_{B_M(y) } G_0( x, \, y) \, dx \geq
\int_{\| y - y_d\|_\infty \leq \tfrac{h_1}{\gamma_1 \, d} }
G_0(x, \, y) \, dx
\geq 
\frac{1}{2} \, 
\left(2 \frac{ h_1}{  \gamma_1\,d} \right)^d.
\]

\item 
For all $n \in \{d, d+1, d+ 2, \ldots\}$
there exist 
$\gamma_1^{(n)}>0$,   $h_1^{(n)}>0$,  $\gamma_2^{(n)}>0$,   $h_2^{(n)}>0$ 
%$\varepsilon >0$ we have 
such that
\begin{equation}
 \label{gtildeprop}
\lim\limits_{n \rightarrow \infty}  
\int\limits_{x\in B_M(y): \, \|x - y\|_\infty \geq  \frac{1}{n} }   \tilde G_0(x, \, y, \, 
\gamma_1^{(n)}, \,  h_1^{(n)}, \,  \gamma_2^{(n)}, \,  h_2^{(n)}   ) \, dx  =0.
%\lim_{\gamma_1 \rightarrow \infty, h_1 \leq  \gamma_1\, \varepsilon,  \, h_2\rightarrow \infty,  \gamma_2  > \frac{d+1}{d} \, h_2 }  
%\int\limits_{y\in B_M: \, \|y - y_{desi}\|_\infty \geq   \varepsilon}   \tilde G_0(y) \, dy  =0.
\end{equation}
\end{enumerate}
Property 3.  can be seen as follows. 
%For $n \in \{1,2,3,\ldots\}$, define 
%\[
%\Omega_n = \{ y \in B_M 
%\mathbb{R}^d
%:
%\|y - y_{desi}\|_\infty \in [  \lambda\, \Delta \, n,   \,   \lambda\, \Delta \, (n+1)) \}
%.
%\]
%
%Then 
%
%
%We have
%\[\int_{\Omega_n}    1 \, dy \leq 
%  (\lambda\, \Delta \, (n+1))^d -  ( \lambda\, \Delta \, n )^d
%  \leq
%    [\lambda\, \Delta \, (n+1)]^d.
%\]
%$
%\lim_{\gamma_1 \rightarrow \infty} 
%\int_{\|y - y_{desi}\|_\infty \geq  \mu \, \lambda \, \Delta  }  
% 4 \, \exp\left( - 2 \, \kappa(\gamma_1) \, h_2 \right).
%$
%
%This can be seen as follows.
Let $n \in \{d, d+1, d+ 2, \ldots\}$
%$\varepsilon >0$
be given.
Define 
%$\lambda \, \Delta = \frac{d}{2}$,
$\gamma_1^{(n)} = n^2$,
 $h_1^{(n)} =   n$, 
$\gamma_2^{(n)} = n$ and
 \[h_2^{(n)} =   \frac{d}{d+1} \,    \left( 1 -    \frac{d}{ \cosh( 2\,  h_1^{(n)} )} \right)  \, \gamma_2^{(n)}.\]

%\varepsilon$.

If   $ \| x -   y\|_\infty \geq \frac{1}{n}   $,
there exists an index  $j \in \{1, \ldots, d\}$ such that
$  |x_j -  y_j|\geq \tfrac{1}{n}    $, and  we have
\[
g_0( \gamma_1^{(n)} (x_j -  y_j) , \, h_1^{(n)}) 
=
\frac{1 + \cosh( 2\, h_1^{(n)})}{ \cosh(  2 \, \gamma_1^{(n)} (x_j -  y_j)   ) + \cosh( 2 \, h_1^{(n)})}
\]
\[
\leq
 \left(  2  + 2\,  \cosh( 2\, h_1^{(n)}) \right)  \, \exp(- 2 \, \gamma_1^{(n)} \,   \tfrac{1}{n}   ) .
\]
%\[
%=
% \frac{1 + \cosh( 2\, h)}{ 2 \cosh( \gamma\,\Delta + h) \, \cosh(\gamma\,\Delta - h   )}.
%\]
%If $h_1 \leq  \gamma_1\, \varepsilon$, 
We have
\[ 1 + \cosh(h_1^{(n)}) \leq 2 \exp(h_1^{(n)})  \leq 2 \exp(   \gamma_1^{(n)} \,   \tfrac{1}{n} ),\]
and obtain
\[
g_0(  \gamma_1^{(n)} \,  (x_j -  y_j) , \, h_1^{(n)})  \leq 4 \exp( -  \gamma_1^{(n)} \,   \tfrac{1}{n} ).
\]
This implies 
\[ d -  \sum_{k = 1}^d g_0( \gamma_1^{(n)}  (x_k -  y_k), \, h_1^{(n)}  )
%\geq  1 - \left(  2  + 2\,  \cosh( 2\, h_1)  \right) \, \exp(- 2 \, \gamma_1\,\Delta).
\geq  1 -  4 \exp( -  \gamma_1^{(n)} \,   \tfrac{1}{n} )  . 
\]
This yields  in turn the inequality
\[G_0(x,    y) \leq 
 g_0 ( \gamma_2^{(n)}  ( d -  \sum_{k = 1}^d g_0(\gamma_1^{(n)}  (x_k -  y_k) )), \, h_2^{(n)} )
 \]
 \[
 \leq \frac{  1 +   \cosh( 2\, h_2^{(n)} ) }{\cosh\left(2 \, \gamma_2^{(n)}  \,   
 \left[   1 -  4 \exp( -  \gamma_1^{(n)}  \,  \tfrac{1}{n}    )   \right] 
 \right) +  \cosh (2\, h_2^{(n)} ) }.
 \]
We have
\[
\int_{B_M(y) }
%{y \in B_M: \|y - y_{desi}\|_\infty \geq  \lambda\, \Delta }  G_0(y) \, dy 
1\, dy 
\leq 
%\sum_{n=1}^M \int_{\Omega_n}    G_0(y) \, dy.
(2M)^d.
\]
 
 Let $  \tilde\varepsilon >0 $ be given such that
 \[  \frac{d}{d+1}  \left( 1 + \tilde\varepsilon \right) < 1. \]
If 
$n$
%$\gamma_1^{(n)} $ 
is sufficiently large we have 
$  1 -  4 \exp( -  \gamma_1^{(n)}  \, \tfrac{1}{n} )  
> \frac{d}{d+1}  \left( 1 + \tilde \varepsilon \right)$ and  thus 
\[
\int_{x \in B_M(y) : \| x - y\|_\infty \geq  \tfrac{1}{n} }  G_0(x, \, y) \, dx
\leq
%\sum_{n=1}^\infty
\frac{  [1 +   \cosh( 2\, h_2^{(n)}   ) ]\, 
(2M)^d
}{\cosh\left(  2\, \frac{d}{d+1}\left( 1 + \tilde\varepsilon \right)\,  \gamma_2^{(n)}  
 \right) +  \cosh (2\, h_2^{(n)} ) }.
\]
%\[
%\leq
%M^d
%\frac{ 1 +   \cosh( 2\, h_2)  }{\cosh\left( 2\, \frac{d}{d+1} \, \gamma_2 
% \right) +  \cosh (2\, h_2) }
%\]
Since 
\[ \gamma_2^{(n)}   > \frac{d+1 }{d} \, h_2^{(n)} , \]
this yields
\[
\int\limits_{x \in B_M(y) : \|x - y\|_\infty \geq  \tfrac{1}{n} }  G_0(x, \, y) \, dx  
%\leq
%M^d
%\frac{ 1 +   \cosh( 2\, h_2^{(n)}   )  }{\cosh\left( 2\, \frac{d}{d+1} \, \gamma_2 
% \right) +  \cosh (2 \, h_2^{(n)} ) }
%\]
%\[
\leq
(2M)^d
\frac{ 1 +   \cosh( 2\, h_2^{(n)} )  }{\cosh\left( 2\, \left( 1 +\tilde \varepsilon \right)\, h_2^{(n)}  
 \right) +  \cosh (2 \, h_2^{(n)} ) }
\]
which converges to zero for $n  \rightarrow \infty $.
%for $h_2 \rightarrow \infty$.
%
With property 2, we obtain the equation 
%\begin{equation}
\[
\lim_{n \rightarrow \infty}  
\int\limits_{x\in B_M(y): \, \|x - y\|_\infty \geq   \tfrac{1}{n}}  \tilde G_0(x, \, y) \, dx  
%\end{equation}
%\[
\leq
\lim_{n \rightarrow \infty} 
  \frac{(2M)^d
  \frac{ 1 +   \cosh( 2\, h_2^{(n)} )  }{\cosh\left( 2\, \left( 1 +\tilde \varepsilon \right)\, h_2^{(n)}  
 \right) +  \cosh (2 \, h_2^{(n)} ) }
 }{
 \frac{1}{2} \, 
\left(\frac{ 2 }{ n\,d} \right)^d
}
=
0
\]
%\end{equation}
and thus we have shown (\ref{gtildeprop}).

%\begin{equation}
%\lim_{\gamma_1 \rightarrow \infty, h_1 \leq  \gamma_1\, \varepsilon,  \, h_2\rightarrow \infty,  \gamma_2  > \frac{d+1}{d} \, h_2 }  
%\int\limits_{y\in B_M: \, \|y - y_{desi}\|_\infty \geq   \varepsilon}  \tilde G_0(y) \, dy  =0.
%\end{equation}

Choose $M$ sufficiently large such that
$K \subset B_{M/2}$.
Extend the function $f$ to a continuous function that is defined on
$B_M(y)$ for all $y \in K$.
For $n \in \{1, \, 2,  \, 3, \, \ldots \}$ 
define
\[\bar U(n)  = \sup_{x, y \in K,\,  \|x-y\|_\infty  \leq \tfrac{1}{n} } |f(x) -  f(y)| <\infty.  \]
%If $\delta$ is sufficiently small,
Then we have
$\lim_{n \rightarrow \infty} \bar U( n ) = 0$.
For all $y\in K$
and  $n \in \{d,\, d+1, \, d+ 2,   \, \ldots \}$ 
we have 
\[ \left|  f(y)  -   \int_{B_M(y)}  f(x) \, \tilde G_0\left(x, \, y, \,\gamma_1^{(n)}, \,  h_1^{(n)}, \,  \gamma_2^{(n)}, \,  h_2^{(n)}    \right) \, dx \right|
\]
\[
=
\left|\int_{B_M(y)} [f(y)  -     f(x)]\, \tilde G_0\left(x, \, y, \,\gamma_1^{(n)}, \,  h_1^{(n)}, \,  \gamma_2^{(n)}, \,  h_2^{(n)}    \right) \, dx \right|
\]
\[
\leq
\left| \int_{B_M(y): \|x-y\|_\infty\leq \tfrac{1}{n} } \tilde G_0(x, \, y ) [ f(y) -  f(x) ]\, dx 
\right|
+
\left| \int_{B_M(y): \|x-y\|_\infty\geq \tfrac{1}{n}} \tilde G_0( x, \, y ) [ f(y) -  f(x) ]\, dx \right|
\]
\[
\leq
\bar U(n)
+
2 \max_{x\in B_M(y)} |f(x)|   \int_{B_M(y): \|x-y\|_\infty \geq  \tfrac{1}{n} } \tilde G_0( y ) \,  dy 
\]
\[
\leq
\bar U(n)
+
2
\max_{\bar y \in K}
\max_{x\in B_M(\bar y)} |f(x)|    \frac{(2M)^d
  \frac{ 1 +   \cosh( 2\, h_2^{(n)} )  }{\cosh\left( 2\, \left( 1 +\tilde \varepsilon \right)\, h_2^{(n)}  
 \right) +  \cosh (2 \, h_2^{(n)} ) }
 }{
 \frac{1}{2} \, 
\left(\frac{ 2 }{ n\,d} \right)^d
}
.
\]
By choosing $n$ sufficiently large,  we can make the first term
arbitrarily small.
%By (\ref{gtildeprop}),
Since also 
%the integral in 
the second term converges to zero 
for
%$\gamma_1$ and $\gamma_2$ are chosen
$n\rightarrow\infty$, the assertion follows.
%This yields the assertion. 

%\end{rem}

%%%%%%%%%%%%%%%%%%%%%%%%%%%%%%%%%%%%%%%%%%%%%%%%%%%%%%%%%%%%%%%%%%%%%%%%%%%%%%%%%%%%%%%%%%

\section{Neural Interpolation}

\label{sec:neuralinterp}
Now we present a constructive solution of the problem of neural interpolation
that is based upon the discussion in Section \ref{sec:auxil}.
We reduce the interpolation problem to the solution of
a system of linear equations.
Using Lemma \ref{lemma2}, we show that the corresponding matrix
is positive definite.

The problem is the following:
Let $N$ 
%desired data 
points $z^{1}, \, \ldots z^{N} \in \mathbb{R}^d$
and
$N$ 
desired data values $z^{desi,1}, \, \ldots z^{desi,N} \in \mathbb{R}^d$
be given.
The task is to find $w^1, \ldots, w^N \in \mathbb{R}^d$ such that  in
the neural map $F: {\mathbb R}^d \rightarrow  {\mathbb R}^d$, 
\[F: \, z\mapsto \sum_{j=1}^N G_0(z, \, z^{j }) \, w^j \]
%such that 
for all $k \in \{1, \ldots, N\}$ we have
the equation
\[F( z^{k}) =  z^{desi,k}.  \]

The following Theorem states that the interpolation problem
can be solved.
\begin{thm}
\label{thm2}
Assume that 
\[\Delta = \min_{j,k \in \{1, \ldots, N\}, \, j\not=k}  \| z^j - z^k\|_\infty>0.\]
Let $\lambda \in (0, \, 1)$ be given and 
let $h_1$ and $h_2$ be defined as in Lemma \ref{lemma2}.

Let $ \varepsilon \in (0,\, 1) $ be given. 
Let $\gamma_1$ and $\gamma_2$ be chosen sufficiently large 
such that for $\kappa(\gamma_1)$
as defined in (\ref{kappadefinition}) we have 
$  \kappa(\gamma_1)>0$  and
\begin{equation}
\label{gershgorinvosaussetzung}
4 \, (N-1)  \, \exp\left( - 2 \, \kappa(\gamma_1)  \, 
 \frac{d}{d+1} \,   \left( 1 -    \frac{d}{ \cosh( 2\,   \lambda \,   \Delta \, \gamma_1)} \right) \, \gamma_2 
\right) < 1 - \varepsilon.
\end{equation}
Then the problem of neural interpolation has a solution,
that is there exist
$w^1, \ldots, w^N \in \mathbb{R}^d$ such that 
for all $k \in \{1, \ldots, N\}$ we have
\begin{equation}
\label{lagrange}
\sum_{j=1}^N G_0(z^k, \, z^{j }) \, w^j  =   z^{desi,k}.
\end{equation}
Moreover, for any
norm on the finite-dimensional space
that contains the parameters 
$(W,  h_1,  h_2,  \gamma_1, \gamma_2)$ 
there exists a constant $C_N$ such that
we can find  a solution 
of the interpolation problem  
that
satisfies the  a priori bound
\begin{equation}
    \label{apriori}
\|(W, \,  h_1,  h_2,  \gamma_1,
\gamma_2)
\|  \leq  C_N ( 1 +  \max_{j\in \{1,\ldots, N\}} \| z^{   desi, \, j}  \|_\infty ).
\end{equation}
Here $W$ is the matrix with the columns $w^1,\, \ldots, w^n$.
\end{thm}
\textbf{Proof:}
The interpolation problem is equivalent to a system of linear equations
for each of the components of the vectors $w^1, \ldots, w^N \in \mathbb{R}^d$.
The quadratic matrix $N\times N$-matrix of these linear systems is
\[P=   \left( G_0( z^k, z^j  )  \right)_{k,j=1 }^N. \]
Since 
$G_0(z^k, \, z^{k })=1$,  the diagonal elements of the matrix $P$ 
are equal to $1$. 

Since $\lambda \in (0, \, 1)$, we can choose 
$\mu = \frac{1}{\lambda} >1$.
For $k\not=j$, 
we have
$
 \| z^j - z^k\|_\infty \geq \Delta  = \mu\, \lambda \,\Delta.
 $
Hence 
due to (\ref{ungleichun202501}) 
we have the inequality
\[
G_0(z^k, \, z^{j })  \leq 4 \,  \exp\left( - 2 \, \kappa(\gamma_1)  \, 
 \frac{d}{d+1} \,   \left( 1 -    \frac{d}{ \cosh( 2\,   \lambda \, \gamma_1 \,  \Delta )} \right) \, \gamma_2 
\right).
\]
Thus due to  (\ref{gershgorinvosaussetzung}),
we can show that  for all $k\in \{1, \ldots, N\}$ we have 
\[ G_0(z^k, \, z^{k }) - \sum_{j\not=k } G_0(z^k, \, z^{j }) > 1 - (1 - \varepsilon) = \varepsilon   , \]
that is the  matrix $P$ is strictly  diagonally dominant,
and hence invertible.
Therefore, there exist $w^1, \ldots, w^N \in \mathbb{R}^d$ that solve the interpolation problem.
Let $W$ denote the $d\times N$ matrix with the columns $w^1, \ldots, w^N$. 

The matrix $P$ is symmetric and  positive definite.
By the Gershgorin circle theorem (see e.g. \cite{varga2011gervsgorin})
the smallest eigenvalue of $P$ is greater than or equal to $\varepsilon$.
Hence for any choice of a norm of the 
finite dimensional space that contains the matrix $W$, 
there exists a constant $\tilde p>0$ 
that does not depend on $W$ or the desired states $z^{desi, \, j}$
such that  we have
%$\|P^{-1}\| \leq \tilde p \, \frac{1}{\varepsilon}$.
%This implies
\[\| W \| \leq  \tilde p \, \frac{1}{\varepsilon}  \max_{j\in \{1,\ldots, N\}}  \| z^{desi, \, j} \|_\infty . \]

The assumptions on 
$h_1$, $h_2$, $\gamma_1$ and $\gamma_2$ imply 
the existence of a constant $\hat C_N>0$ such that we have
$\|(h_1,  h_2,  \gamma_1, \gamma_2 ) \|  \leq  \hat C_N$.
Thus there exist a constant $\hat p$ such that we have 
$$\|(W,  h_1,  h_2,  \gamma_1,
\gamma_2
\|
\leq \hat p
(\|(h_1,  h_2,  \gamma_1, \gamma_2) \| + \|W\|)
\leq
\hat p \left(  \hat C_N +  \frac{ \tilde p}{\varepsilon}   \max_{j\in \{1,\ldots, N\}}  \,\| z^{desi, \, j}  \|_\infty \right),
$$
and
(\ref{apriori}) follows.
This finishes the proof of Theorem \ref{thm2}.

%%%%%%%%%%%%%%%%%%%%%%%%%%%%%%%%%%%%%%%%%%%%%%%%%%%%%%%%%%%%%%%%%%%%%%%%%%%%%%%%%%%%%%%%%%%%%%%%%%%%%%%%%%%%%%%%%

%%%%%%%%%%%%%%%%%%%%%%%%%%%%%%%%%%%%%%%%%%%%%%%%%%%%%%%%%%%%%%%%%%%%%%%%%%%%%%%%%%%%%%%%%%%%%%%%%%%%%%%%%

\begin{rem}
\label{remark2}
The proof of (\ref{ungleichun202501})  shows that  If   $ \| y -   y^{desi}\|_\infty \geq \mu \, \lambda \, \Delta$, we have 
\[
 G_0(y,    y^{desi})
 \leq \frac{  1 +   \cosh( 2\, h_2) }{\cosh\left(2 \, \gamma_2 \,   
 \left[  1 -  4 \exp( - 2 (\mu - 1)  \,  \lambda  \, \gamma_1 \, \Delta) \right] 
 \right) +  \cosh (2 \, h_2) }.
 \]
 If
$\left[  1 -  4 \exp( - 2 (\mu - 1)  \,  \lambda  \, \gamma_1 \, \Delta) \right] >0$,
 this implies that asymptotically,  for $\gamma_2\rightarrow \infty$, 
%and $\gamma_2\rightarrow \infty$,
we have
\[
w^k(\gamma_2) \rightarrow  z^{desi,k}.
\]
This follows 
%from (\ref{ungleichun202501}), 
since
asymptotically, the non-diagonal elements 
 of the matrix $P$  converge to zero
 exponentially fast 
while the diagonal elements are equal to $1$.
\end{rem}

\section{Exact ensemble control}

\label{sec:exact}

We show the exact ensemble controllability by
constructing a time-varying feedback law
of the type that has been studied already
%be {J,-M. Coron} (see e.g. 
in \cite{CORON1995159}.
%).
%and {\sc Emmanuel Tr\'elat}.
%
The controls in our system are
in particular the matrices $W^{(i)}(t)$,
whereas we consider
$\gamma_1$
and $\gamma_2$
(and the corresponding values of 
$h_1$ and $h_2$ as defined in Lemma \ref{lemma2})
as parameters
that are chosen depending on
$d$, $N$ and $ \lambda \, \Delta$, but not on the realization  of the desired trajectories.
%$z^{desi}$

Let a time horizon $T$ and  $N$ desired trajectories
\[y^{desi, j }(t), \; t\in [0, T], \; j \in \{1,\ldots, N\}\]
be given.
Our aim is to find control parameters such that
our system tracks the $N$ 
desired trajectories 
$y^{desi, j }$
simultaneously.
%in a stable way.
%
In particular, the $N$ initial states
are steered exactly to the $N$ terminal states
at the time $T$.

To make this is  possible 
we assume that the trajectories
are chosen in such a way that they do not  have  intersection points
in the interior of $[0, \, T)$.
Note that for $d\geq 3$, in $\mathbb{R}^d$ 
generically, trajectories do not intersect
and such a choice is possible.

If there are common target states,
the time-reversability of the system is lost 
but also in this case we can 
track the trajectories in
such a way that they unify at the 
time $T$.
%
%
%
%see 
%\cite{dahinden2026intersectionsprojectedhamiltonianorbits}.
%Lucas Dahinden (Utrecht)
%We go on to prove that generically in dimension greater than 2, chords between two points do not have self-intersections at all, generalizing a theorem by Rademacher.
%This is joint work with Jacobus de Pooter.
%On the intersections of projected Hamiltonian orbits in cotangent bundles
%
%
 Now we state our main result:
 \begin{thm}
 \label{thm3}
 Let $d \geq 3$.
Assume that $N$ different 
initial states 
$y^{desi, j }(0)$
and
$N$ different
target states
$y^{desi, j }(T)$
are given.
%desired trajectories 
%$y^{desi, j }$ do not intersect in $[0, \, T]$.
 
Then the  system governed by the neural ODE 
(\ref{ode}) with $ F$ as in  (\ref{fdefinition}),
depth $d_e\geq 2$ and  $F^{(i)}$ as in
(\ref{fidefinition})
is exactly ensemble controllable
if in $F^{(2)}$ 
the matrix $W^{(2)}$ has $2N$ columns 
%has width $2N$ 
and  in $F^{(1)}$ 
the matrix $W^{(1)}$ has $2d$ columns.
%has width $2d$.
To be precise, we have the following assertion: 
%in the following sense:

There exist controls
$(W(t), \, h_1, \, h_2, \, \gamma_1, \,\gamma_2)$
(where $W(t) = (W^{(1)},  W^{(2)}(t))   $ 
and  $ W^{(2)}(t) $  is a  continuously differentiable  function) 
such that 
for all $j\in \{1,\ldots, N\}$ the initial states $y^{desi, j }(0)$
are steered to the target states $y^{desi, j }(T)$
at the time $T$.

 For any
norm on the finite-dimensional space
that contains the parameters 
$(W(t) ,  h_1,  h_2,  \gamma_1, \gamma_2)$ 
there exists a constant $C_N$ such that
we can chose the controls in such a way that
%of the interpolation problem  that satisfies the  a priori bound
%Moreover,  the control functions can be chosen in such a way that 
for all $t\in [0, \, T]$ 
they satisfy the  a-priori bound
\begin{equation}
    \label{aprioristeuerbar}
\|(W(t),  h_1,  h_2,  \gamma_1, \gamma_2)
\|  \leq  C_N ( 1 + 
%\max_{t\in [0, \, T]}  
\max_{j\in \{1, \ldots, N\} } \| y^{   desi, j}(t)  \|_{ \color{red} C^1([0, T])}
%\infty 
),
\end{equation}
where 
$ y^{   desi, j}(t) $ are suitably chosen
continuously differentiable 
desired trajectories that are pairwise disjoint.

%The system has the solutions
%$y^{desi, j }(t)$,
%$t\in [0, \, T]$,
%$j \in \{1, \ldots, N\}$.

Note that the controls
$W^{(1)}$, $A^{(i)}$   and  $b^{(i)}$ 
($i \in \{1,\, 2\}$) 
in (\ref{fidefinition}) 
can be chosen in such a way that they do not depend on the time $t$.

 \end{thm}
\begin{rem}
The result also holds for $d=2$ if there exist
desired trajectories 
$ y^{   desi, j}(t) $
that are pairwise disjoint.

In  Theorem \ref{thm3}, the function $F$ has the form
\[F(t,\, y) = \sum_{j=1}^N G_0(y, \, y^{desi, \, j }(t))  \, w^j(t) \]
with $G_0$ as in (\ref{g0definition}).

\end{rem}
 
\textbf{Proof:} 
To obtain a suitable control,
we use neural interpolation
as 
constructed in Theorem \ref{thm2}.
%Section \ref{sec:neuralinterp}
%(see also \cite{doi:10.1137/23M1599744}).
Let $N$ 
desired trajectories 
\[y^{desi, j }(t), \;
 t\in [0, \, T],  \; j \in \{1,\ldots, N\}
 \]
be given that are continuously differentiable. 
Since $d\geq 3$,
we can assume without loss of generality that
there exists a number $\Delta>0$ such that  for all 
$t\in [0,\, T]$ and all $j\not=k$ we have
\[
\| y^{desi, j }(t) -  y^{desi, k }(t)\|_\infty \geq 2\, \Delta.
\]
Choose $\lambda = \frac{1}{2} $ and $\mu = 2$ and let $ \varepsilon \in (0,\, 1) $ be given. 
For all $t\in [0, \, T]$, we solve the interpolation problem
\[
F(t,  \,  y^{desi, j}(t)) =   (y^{desi, \, j})'(t) - L  \,  y^{desi, j}(t))  ,
\; j \in \{1,\ldots, N\}. 
\]
To solve this problem, for all
$t\in [0, \, T]$, we determine the vectors 
$w^k(t) \in \mathbb{R}^d$, $k \in \{1, \ldots, N\} $
as solutions of the linear interpolation problem
\[ (y^{desi, \, k})'(t) - L   y^{desi, k}(t))   =  \sum_{j=1}^N  G_0( y^{desi, \, k}(t), \,  y^{desi, j}(t)  ) \, w^k(t) , \; k \in \{1, \ldots, N\}
\]
with $w^k(t) \in \mathbb{R}^d$ for $k \in \{1,\ldots, N\}$.
For this purpose, the parameters $\gamma_1$, $h_1$, $\gamma_2$, $h_2$
are chosen as in Theorem  \ref{thm2},
that is
%2404
 $\gamma_1$ and $\gamma_2$ are chosen sufficiently large 
such that  $\kappa(\gamma_1)>0$
(with $\kappa$  as defined in (\ref{kappadefinition}))
and (\ref{gershgorinvosaussetzung}) holds
and
$h_1$ and $h_2$ are defined as in Lemma \ref{lemma2}.

The linear system for the
$d$ components of the vectors $w^1(t), \ldots, w^N(t)$
has the matrix
\[P(t) = ( G_0( y^{desi, \, k}(t), \,  y^{desi, j}(t)  ) )_{k,\, j=1}^N.
\]
By our choice of the parameters $\gamma_1$, $\gamma_2$, $h_1$, $h_2$
for all $t\in [0, \, T]$
the smallest eigenvalue of $P(t)$ is greater than or equal to $\varepsilon$.
Hence for any choice of a norm of the 
finite dimensional space that contains the matrices $W(t)$, 
there exists a constant $\tilde p>0$
that does not depend on $t$ such that  we have
%$\|P^{-1}\| \leq \tilde p \, \frac{1}{\varepsilon}$.
%This implies
\[\| W(t) \| \leq  \tilde p \, \frac{1}{\varepsilon}  \max_{j\in \{1,\ldots, N\}}  \| y^{desi, \, j}(t) \|_{ \color{red} C^1([0, \, T])}
%_\infty 
. \]
 We have  $G_0( y^{desi, \, j}(t), \,  y^{desi, j}(t)  )=1$
and also the other entries 
%of $P(t)$  depend continuously  on $t$.
%Since the values 
of $P(t)$  depend  in a continuously differentiable way on $t$.
Thus  also  the solution 
$(w^k(t))_{k=1}^N$
of the interpolation problem depends
 continuously differentiable on $t$.
 By our construction, the a priori bound
(\ref{apriori}) implies  (\ref{aprioristeuerbar}).
This finishes the proof of Theorem \ref{thm3}.

 %\textbf{The question arises: Is it also possible for $d_e=1$?}
%
% To obtain a good approximate solution, we can choose
% \[w_k(t) = y^{desi, \, k}(t).\]
%(This is a similar strategy as  in polynomial interpolation with  Lagrange basis polynomials.
%However, a fundamental difference is that in our case, 
%we have basis functions that are positive,
%similarly as in the case of Bernstein polynomials.
%Thus this procedure gives rise to a converging process
%for approximation for a continuous field of data.)

 %Then   %by  (\ref{errorbound}), 
 %the interpolation error is  of the order $\exp(- \gamma_2)$, 
 %so if we choose $\gamma>0$
 %sufficiently large,
 %the interpolation error can be made arbitrarily small.
 %In this way, we obtain an explicit representation
 %of approximate ensemble controls towards trajectories,
 %that requires neither optimization
 %nor the solution of a system of equations,
 %thus saving computing ressources.

%%%%%%%%%%%%%%%%%%%%%%%%%%%%%%%%%%%%%%%%%%%%%%%%%%%%%%%%%%%%%%%%%%%%%%%%%%%%%%%%%%%%%%%%%%

\section{Example: An ensemble with two elements} % in dimension 2}
\label{sec:exa}

In order to illustrate the construction,
we consider the case $N=2$ in detail.
For brevity, for the operator $L$ in (\ref{ode})   we  assume that $L=0$.

For $N=2$, substituting  $\lambda = \frac{1}{2}$ and $\mu = 2$ into  
(\ref{kappadefinition}) 
yields the condition that $\gamma_1$ is sufficiently large such that 
\begin{equation}
\label{kappadefinitionwdh}
\kappa(\gamma_1) = 
\frac{d+1}{d}  \frac{1}{ 1 -    \frac{d}{ \cosh(   \Delta \, \gamma_1) } }
 \left[
 1 -  4 \exp( - \Delta \, \gamma_1)
 \right]
-1 > 0
\end{equation}
and (\ref{gershgorinvosaussetzung})    requires that $\gamma_2$ is sufficiently large such that 
\begin{equation}
\label{gershgorinvosaussetzungwdh}
4 \,
%(N-1)  \, 
\exp\left( - 2 \, \kappa(\gamma_1)  \, 
 \frac{d}{d+1} \,   \left( 1 -    \frac{d}{ \cosh(    \Delta \, \gamma_1)} \right) \, \gamma_2 
\right) < 1 - \varepsilon.
\end{equation}

For the two desired trajectories,  we use the notation
\[y^{desi, \, 1}(t)
=
\begin{pmatrix}
   y^{desi, \, 1}_1(t)
\\
    y^{desi, \, 1}_2(t)
\end{pmatrix},
\;\;
y^{desi, \, 2}(t)
=
\begin{pmatrix}
   y^{desi, \, 2}_1(t)
\\
    y^{desi, \, 2}_2(t)
\end{pmatrix}
.
\]
We have the $2\times2$ matrix
\[
P(t) =
\begin{pmatrix}
1 & G_0(y^{desi, \, 1}(t),  \, y^{desi, \, 2}(t)  )
\\
G_0(y^{desi, \, 1}(t),  \, y^{desi, \, 2}(t)  ) & 1
\end{pmatrix}
\]
with the determinant
$\det P(t) =  1 - [G_0(y^{desi, \, 1}(t),  \, y^{desi, \, 2}(t)  )]^2 $.

Thus we have the inverse matrix
\[
[P(t)]^{-1} = \frac{1}{\det P(t)}
\begin{pmatrix}
1 &  -G_0(y^{desi, \, 1}(t),  \, y^{desi, \, 2}(t)  )
\\
-G_0(y^{desi, \, 1}(t),  \, y^{desi, \, 2}(t)  ) & 1
\end{pmatrix}
.
\]
For the controls
\[w^1(t)
=
\begin{pmatrix}
   w^1_1(t)
\\
   w^1_2(t)
\end{pmatrix},
\;\;
w^2(t)
=
\begin{pmatrix}
 w^2_1(t)
\\
   w^2_2(t)
\end{pmatrix}
\]
we obtain the explicit representation
\[
\begin{pmatrix}
   w^1_1(t)
\\
   w^2_1(t)
\end{pmatrix}
=
[P(t)]^{-1}
\begin{pmatrix}
  ( y^{desi, \, 1}_1)'(t)
\\
  (  y^{desi, \, 2}_1)'(t)
\end{pmatrix},
\;
\begin{pmatrix}
   w^1_2(t)
\\
   w^2_2(t)
\end{pmatrix}
=
[P(t)]^{-1}
\begin{pmatrix}
   (y^{desi, \, 1}_2)'(t)
\\
    (y^{desi, \, 2}_2)'(t)
\end{pmatrix}.
\]

Then we have 
\[F(t,\, y) = G_0(y, \, y^{desi, \, 1}(t))  \, w^1(t)  + G_0(y, \, y^{desi, \, 2}(t))  \, w^2(t).
\]

For a specific example, consider
\[y^{desi, \, 1}(t)
=
\begin{pmatrix}
   y^{desi, \, 1}_1(0) + t
\\
    y^{desi, \, 1}_2(0)
\end{pmatrix},
\;\;
y^{desi, \, 2}(t)
=
\begin{pmatrix}
   y^{desi, \, 2}_1(0) - t
\\
    y^{desi, \, 2}_2(0)
\end{pmatrix}
\]
with the given initial data 
\[
\begin{pmatrix}
   y^{desi, \, 1}_1(0) 
\\
    y^{desi, \, 1}_2(0)
\end{pmatrix}
\not =
\begin{pmatrix}
   y^{desi, \, 2}_1(0) 
\\
    y^{desi, \, 2}_2(0)
\end{pmatrix}
.
\]
We assume that for all $t\in [0, \, T]$ we have
$y^{desi, 1 }(t) \not =   y^{desi, 2 }(t)$.

%there exists a number $\Delta>0$ such that  for all 
%$t\in [0,\, T]$  we have
%\[
%\| y^{desi, 1 }(t) -  y^{desi, 2 }(t)\|_\infty \geq 2\, \Delta.\]

Then we have 
\[
\begin{pmatrix}
   w^1_1(t)
\\
   w^2_1(t)
\end{pmatrix}
=
[P(t)]^{-1}
\begin{pmatrix}
1
\\
-1
\end{pmatrix}
=
\frac{  1  + G_0(y^{desi, \, 1}(t),  \, y^{desi, \, 2}(t)  )  }{ \det P(t) }
\begin{pmatrix}
1
\\
- 1
\end{pmatrix},
\]
\[
\begin{pmatrix}
   w^1_2(t)
\\
   w^2_2(t)
\end{pmatrix}
=
[P(t)]^{-1}
\begin{pmatrix}
   0
\\
  0
\end{pmatrix}
=
\begin{pmatrix}
   0
\\
  0
\end{pmatrix}
\]
Thus we have
\[w^1(t)
=
\begin{pmatrix}
 \frac{  1  + G_0(y^{desi, \, 1}(t),  \, y^{desi, \, 2}(t)  )  }{  1 - [G_0(y^{desi, \, 1}(t),  \, y^{desi, \, 2}(t)  )]^2  }
\\
  0
\end{pmatrix},
\;\;
w^2(t)
=
\begin{pmatrix}
-  \frac{  1  + G_0(y^{desi, \, 1}(t),  \, y^{desi, \, 2}(t)  )  }{   1 - [G_0(y^{desi, \, 1}(t),  \, y^{desi, \, 2}(t)  )]^2  }
\\
  0
\end{pmatrix}
\]
and obtain
\[w^1(t)
=
\begin{pmatrix}
 \frac{ 1 }{   1  - G_0(y^{desi, \, 1}(t),  \, y^{desi, \, 2}(t)  )   }
\\
  0
\end{pmatrix},
\;\;
w^2(t)
=
\begin{pmatrix}
-  \frac{  1   }{   1  - G_0(y^{desi, \, 1}(t),  \, y^{desi, \, 2}(t)  )  }
\\
  0
\end{pmatrix}.
\]

Then  for $F = \begin{pmatrix} F_1 \\ F_2 \end{pmatrix}$ we have
$F_1(t,\, y) = 
\frac{   G_0(y, \, y^{desi, \, 1}(t))  - G_0(y, \, y^{desi, \, 2}(t))     }{   1  - G_0(y^{desi, \, 1}(t),  \, y^{desi, \, 2}(t)  )   }
$
and
$F_2(t,\, y) = 0  $.

%\[F_1(t,\, y^{desi, \, 1}(t) ) = 1, \, F_1(t,\, y^{desi, \, 1}(t) ) = -1.\]
%\frac{   G_0(y, \, y^{desi, \, 1}(t))  - G_0(y, \, y^{desi, \, 2}(t))     }{   1  - G_0(y^{desi, \, 1}(t),  \, y^{desi, \, 2}(t)  )   }
%\]

For $ N = 3$,   we have the $3\times3$ matrix
\[
P(t) =
\begin{pmatrix}
1 & G_0(y^{desi, \, 1}(t),  \, y^{desi, \, 2}(t)  ) & G_0(y^{desi, \, 1}(t),  \, y^{desi, \, 3}(t)  )
\\
G_0(y^{desi, \, 1}(t),  \, y^{desi, \, 2}(t)  ) & 1 & G_0(y^{desi, \, 2}(t),  \, y^{desi, \, 3}(t)  ) 
                                                 \\               
G_0(y^{desi, \, 1}(t),  \, y^{desi, \, 3}(t)  )  & G_0(y^{desi, \, 2}(t),  \, y^{desi, \, 3}(t)  )  & 1
\end{pmatrix}.
\]
Define the set $I=\{(1,2),(1,3),(2,3)\} $.
Then for the determinant we have 
\[\det P(t) = 1 
+ 2 \Pi_{(i,j) \in I}  G_0(y^{desi, \, i}(t),  \, y^{desi, \, j}(t)  ) 
-  \sum_{(i,j) \in I} 
|G_0(y^{desi, \, i}(t),  \, y^{desi, \, j}(t)  )|^2 
%- | G_0(y^{desi, \, 1}(t),  \, y^{desi, \, 3}(t)  )|^2 -  
%|G_0(y^{desi, \, 2}(t),  \, y^{desi, \, 3}(t)  ) |2
%+ 2   G_0(y^{desi, \, 1}(t),  \, y^{desi, \, 2}(t)  ) \,  G_0(y^{desi, \, 1}(t),  \, y^{desi, \, 3}(t)  ) \, G_0(y^{desi, \, 2}(t),  \, y^{desi, \, 3}(t)  ) 
\]
and  with the notation
\[
P(t) =
\begin{pmatrix}
1 & a  & b
\\
a & 1 & c
\\               
b  & c  & 1
\end{pmatrix}.
\]
the inverse is
\[
P(t)^{-1}=
\frac{1}{\det P(t)}
\begin{pmatrix}
1 -  c^2 &  -a + b\, c  &   - b + a \, c
\\
-a +  b\, c  &  1 - b^2 & -c +  a \, b 
 \\
-b +  a\, c  & -c +    a\, b  &  1 - a^2
\end{pmatrix}.
\]
This suggests to use as a first order approximation for  $P(t)^{-1}$
the matrix
\[M(t) =
\begin{pmatrix}
1  &  -a &   - b 
\\
-a  &  1 & -c
 \\
-b   & -c  &  1 
\end{pmatrix}.
\]
We have
\[
M(t) \, P(t) =\begin{pmatrix}
1-a^2 - b^2  &  - b c &   - c a 
\\
- b c   &  1 - a^2 - c^2  & - a b 
 \\
- a c   & - a b  &  1 - b^2 - c^2 
\end{pmatrix},
\]
which is, for small $a$, $b$, $c$, quadratically close
to the identity matrix $I$. 
With the corresponding ansatz for $M(t)$, this also generalizes to
higher dimensions.

%that has the Cholesky decomposition
%$P(t) = L(t) \,L(t)^\top $,
%with
%\[L(t) =
%\begin{pmatrix}
%1 & 0 & 0 
%\\
%G_0(y^{desi, \, 1}(t),  \, y^{desi, \, 2}(t)  ) & \sqrt{ 1 - |G_0(y^{desi, \, 1}(t),  \, y^{desi, \, 2}(t)  )  |^2} & 0
%\\
%G_0(y^{desi, \, 1}(t),  \, y^{desi, \, 3}(t)  ) & 
%\frac{  G_0(y^{desi, \, 2}(t),  \, y^{desi, \, 3}(t)  )  - G_0(y^{desi, \, 1}(t),  \, y^{desi, \, 2}(t)  )  G_0(y^{desi, \, 1}(t),  \, y^{desi, \, 3}(t)  )}{
%\sqrt{ 1 - |G_0(y^{desi, \, 1}(t),  \, y^{desi, \, 2}(t)  )  |^2} } &
%*
%\end{pmatrix}
%\]

%%%%%%%%%%%%%%%%%%%%%%%%%%%%%%%%%%%%%%%%%%%%%%%%%%%%%%%%%%%%%%%%%%%%%%%%%%%%%%%%%%%%%%%%%%
\section{Conclusion}
We have shown that for 
neural differential equations with sufficiently large neural networks, 
exact ensemble controllability is possible.
Our construction does not require optimization,
but only the solution of a linear  system of equations.
Our result allows further developments of the
mathematical theory of the optimal control of
neural ODEs, for example,
it makes it possible to show that the optimal controls
satisfy a turnpike property for large time-horizons,
see for example  \cite{faulwasser2024turnpike}, \cite{puttschneider2025towards}, \cite{gugat2024finite}, 
\cite{gugat2025optimal}.
%Our construction depends on neural networks of depth two.
Note that the present construction relies on the assumption that the desired trajectories are pairwise disjoint, which is generically satisfied for 
$d \geq 3$,  but requires additional care for  $d = 2$.
One possibility to avoid this difficulty is to add an auxiliary dimension to
augment the space artificially to dimension $3$.
%%%%%%%%%%%%%%%%%%%%%%%%%%%%%%%%%%%%%%%%%%%%%%%%%%%%%%%%%%%%%%%%%%%%%%%%%%%%%%%%%%%%%%%%%%%%%%%%%%%%%%%%%%%%%%%%%%%%%%%%%%%%%%%%%%%%%%%

Several directions for future research remain open. 
%
%First, the present construction relies on the assumption that the desired trajectories are pairwise disjoint, which is generically satisfied for 
%$d \geq 3$ 
%but requires additional care for 
%$d = 2$; a systematic treatment of the planar case would be of interest. Second, 
While we have focused on the hyperbolic tangent as the activation function, 
it would be worthwhile to extend the neural interpolation framework to other sigmoidal 
or ReLU-type activations. 
We expect that for networks of depth two, an analogous  construction of
localized functions $G_0$ is possible, but 
in the case of non-smooth activation function with less regularity.
The a priori bounds established in Theorem \ref{thm3}  suggest that the dependence of the control cost 
constant $C_N$  on $N$ and $d$ could be made explicit, 
which would be relevant for applications where the ensemble size is large. 
Finally, combining the exact controllability result with turnpike theory - as indicated above -
appears to be a promising avenue for establishing structural properties of optimal control problems for neural ODEs over 
long time horizons. For large $N$, our results could also allow to analyze the turnpike property for
a mean-field approach for problems with neural coupling terms, see \cite{gugat2024turnpike}.

%%%%%%%%%%%%%%%%%%%%%%%%%%%%%%%%%%%%%%%%%%%%%%%%%%%%%%%%%%%%%%%%%%%%%%%%%%%%%%%%%%%%%%%%%%%%%%%%%%%%%%%%%%%%%%%%%%%%%%%%%%%%%%%%%%%%%%%

%%%%%%%%%%%%%%%%%%%%%%%%%%%%%%%%%%%%%%%%%%%%%%%%%%%%%%%%%%%%%%%%%%%%%%%%%%%%%%%%%%%%%%%%%%%%%%

\textbf{Acknowledgements:} This work is supported by 
the Deutsche Forschungsgemeinschaft (DFG) 
in the
Collaborative Research Centre CRC/Transregio 154, Mathematical Modelling,
Simulation and Optimization Using the Example of Gas Networks, Project C03
and Project C05, Projektnr. 239904186 and 
the project 
\emph{Regelung inter\-agierender Partikelsysteme und ihre probabilistischen und fluid-dynamischen Be\-schreibungen}, Projektnummer: 560288187.
%%%%%%%%%%%%%%%%%%%%%%%%%%%%%%%%%%%%%%%%%%%%%%%%%%%%%%%%%%%%%%%%%%%%%%%%%%%%%%%%%%%%%%%%%%%%%%%

\bibliography{artikel}{}
\bibliographystyle{plain}
%\bibdata{artikel}

\end{document}